\documentclass[11pt]{amsart}
\usepackage[dvipsnames]{xcolor}
\usepackage[pagebackref,backref=true,colorlinks=true, linkcolor=Sepia, citecolor=Sepia]{hyperref}
\usepackage{amsmath,amsthm,amssymb,fancyhdr,graphicx,bbm,cancel,mathrsfs,todonotes,xypic,pinlabel}
\usepackage{cleveref}
\usepackage{tikz}
\usetikzlibrary{external}
\usepackage{extarrows,tipa}
\usetikzlibrary{matrix}
\usepackage[all]{xy}
\usepackage{mathtools}
\usepackage{bm}
\usepackage{verbatim}
\usepackage{microtype}
\usepackage{subcaption}
\usepackage{tikz}
\usepackage{tikz-cd}
\theoremstyle{theorem}
\newtheorem{theorem}{Theorem}[section]
\newtheorem{theoremalpha}{Theorem}

\newtheorem{lemma}[theorem]{Lemma}

\newtheorem{corollary}[theorem]{Corollary}
\theoremstyle{definition} 

\newtheorem{remark}[theorem]{Remark}
\theoremstyle{remark} 

\usepackage{enumerate,pdfpages,stmaryrd}
\usepackage[margin=1in, marginparwidth=1in]{geometry}
\usepackage{tikz}
\usepackage{dsfont}
\usetikzlibrary{matrix}

\renewcommand{\xto}{\xrightarrow}

\newcommand{\C}{\mathbb{C}}

\newcommand{\Z}{\mathbb{Z}}

\newcommand{\R}{\mathbb{R}}

\newcommand{\F}{\mathbb{F}}

\newcommand{\beq}{\begin{equation*}}
\newcommand{\eeq}{\end{equation*}}

\newcommand{\On}{\mathrm{O}}

\newcommand{\GL}{\mathrm{GL}}

\DeclareMathOperator{\Diff}{Diff}

\DeclareMathOperator{\Aut}{Aut}

\DeclareMathOperator{\id}{id}

\DeclareMathOperator{\Top}{Top}

\newcommand{\dmo}{\DeclareMathOperator}

\newcommand{\N}{\mathbb{N}}

\newcommand{\Om}{\Omega}\newcommand{\ka}{\kappa}

\newcommand{\wtil}{\widetilde}

\newcommand{\bb}[1]{\mathbb{#1}}

\dmo{\sgn}{sign}\dmo{\Span}{span}
\dmo{\we}{\wedge}
\dmo{\ind}{ind}\dmo{\Ind}{Ind}
\dmo{\bop}{\bigoplus}\dmo{\pic}{Pic}
\dmo{\vol}{Vol}\dmo{\gal}{Gal}\dmo{\perm}{Perm}
\dmo{\tor}{Tor}\dmo{\ext}{Ext}\dmo{\Ext}{Ext}
\dmo{\aut}{aut}

\dmo{\inn}{Inn}\dmo{\var}{Var}
\dmo{\ad}{ad}\dmo{\curl}{curl}
\dmo{\hy}{\bb H}\dmo{\Sl}{SL}
\dmo{\psl}{PSL}
\dmo{\iso}{iso}
\dmo{\conf}{Conf}
\dmo{\stab}{Stab}\dmo{\Jac}{Jac }
\dmo{\diam}{diam}\dmo{\fix}{Fixed}\dmo{\Fix}{Fix}
\dmo{\injR}{injRad}\dmo{\Ad}{Ad}
\dmo{\esv}{ess-vol}
\dmo{\nil}{Nil}\dmo{\sol}{Sol}
\dmo{\Div}{div}
\dmo{\SU}{SU}
\dmo{\rk}{rk}
\dmo{\rank}{rank}
\dmo{\psp}{PSp}\dmo{\psu}{PSU}
\dmo{\PU}{PU}\dmo{\pgl}{PGL}
\dmo{\Mod}{Mod}\dmo{\range}{Range}
\dmo{\eu}{eu}\dmo{\mi}{mi}
\dmo{\Log}{Log}\dmo{\supp}{supp}
\dmo{\maps}{Maps}\dmo{\Gr}{Gr}
\dmo{\Pin}{Pin}
\dmo{\Spin}{Spin}\dmo{\Str}{Str}
\dmo{\Sq}{Sq}\dmo{\Symp}{Symp}
\dmo{\pd}{PD}\dmo{\PD}{PD}\dmo{\sig}{Sig}
\dmo{\ev}{ev}\dmo{\St}{St}
\dmo{\Pt}{Pt}\dmo{\pt}{pt}
\newcommand{\msf}{\mathsf}

\dmo{\Pl}{PL}
\dmo{\String}{String}\dmo{\smear}{smear}

\dmo{\dev}{dev}
\dmo{\met}{Met}\dmo{\contact}{Contact}
\dmo{\teich}{Teich}\dmo{\Teich}{Teich}\dmo{\qi}{QI}
\dmo{\der}{Der}
\dmo{\cl}{Cliff}\dmo{\Cl}{Cl}
\dmo{\Pf}{Pf}
\dmo{\ch}{ch}\dmo{\diag}{diag}
\dmo{\grad}{grad}\dmo{\Char}{char}
\dmo{\spec}{Spec}\dmo{\Arg}{Arg}
\dmo{\gl}{GL}
\dmo{\sym}{Sym}\dmo{\Sym}{Sym}
\dmo{\com}{Comm}
\dmo{\Lk}{Lk}
\dmo{\CAT}{CAT}
\dmo{\Rep}{Rep}
\dmo{\Res}{Res}
\dmo{\Conf}{Conf}
\dmo{\PConf}{PConf}
\dmo{\Push}{Push}
\dmo{\Cont}{Cont}
\dmo{\sm}{\setminus}
\dmo{\vn}{\varnothing}
\dmo{\disk}{\mathbb D}
\dmo{\Trd}{Trd}\dmo{\Mat}{Mat}
\dmo{\Riem}{Riem}
\dmo{\Diffn}{\Diff_0}\dmo{\diff}{diff}
\dmo{\homeo}{Homeo}

\dmo{\Ham}{Ham}\dmo{\Met}{Met}
\dmo{\Ein}{Ein}\dmo{\CP}{\co P}
\dmo{\Per}{Per}\dmo{\Ric}{Ric}
\dmo{\Nrd}{Nrd}
\dmo{\Comp}{Comp}\dmo{\PSC}{PSC}
\dmo{\Cent}{Cent}\dmo{\Orb}{Orb}
\dmo{\aind}{a-ind}\dmo{\tind}{t-ind}
\dmo{\constant}{constant}
\dmo{\Td}{Td}
\dmo{\LMod}{LMod}
\dmo{\SMod}{SMod}
\dmo{\SDiff}{SDiff}
\dmo{\Br}{Br}
\dmo{\csch}{csch}
\dmo{\triv}{triv}
\dmo{\genus}{genus}
\dmo{\Homeq}{HomEq}
\dmo{\PP}{\mathbb{P}}
\dmo{\U}{U}
\dmo{\Gal}{Gal}
\dmo{\BDiff}{\wtil{\Diff}}
\dmo{\BAut}{\wtil{\Aut}}
\dmo{\Iso}{Iso}
\dmo{\Cone}{Cone}
\dmo{\codim}{codim}
\dmo{\II}{II}
\dmo{\I}{I}
\dmo{\InjRad}{InjRad}
\dmo{\Inn}{Inn}
\dmo{\sys}{sys}
\dmo{\Comm}{Comm}
\dmo{\PO}{PO}
\dmo{\vertex}{Vert}
\dmo{\POm}{P\Om}
\dmo{\ab}{ab}
\dmo{\PSO}{PSO}
\dmo{\CRS}{CRS}
\dmo{\Diffext}{Diffext}
\dmo{\Diffextad}{Diffextad}
\dmo{\Diffstand}{Diffstand}

\begin{document}
\title[]{Aspherical manifolds with nonvanishing\\ tautological classes}
\author{Mauricio Bustamante}
\maketitle
\begin{abstract}
For every even integer $m\geq 2$, we construct a closed, orientable, smooth, aspherical $(2m+1)$-manifold whose fundamental group has nontrivial center and for which infinitely many tautological classes in the $\F_2$-cohomology of the classifying space of its group of homotopically trivial diffeomorphisms are nonzero and not nilpotent. Specializing the calculation to cohomological degree $0$ gives examples of such manifolds that do not bound compact smooth manifolds. These provide counterexamples to a conjecture of Hebestreit--Land--L\"uck--Randal-Williams.
\end{abstract}
\section{Introduction}
A smooth manifold $M$ is called \textit{aspherical} if its universal cover is contractible. Hence these manifolds are Eilenberg--Maclane spaces and so much of their topology is encoded in their fundamental groups. In this regard, there seems to be some connection between the role played by the \textit{center} of the fundamental group $\pi_1(M)$ of $M$ and the vanishing of its topological invariants. Gottlieb \cite[Corollary IV.3]{gottlieb} showed that the Euler characteristic of a closed aspherical manifold with nontrivial center vanishes. In the same vein, Farrell \cite[Theorem 1.2]{farrell} showed that the signature of a closed orientable aspherical manifold of dimension $4k$ vanishes if, in addition to having nontrivial center, the fundamental group is residually finite. Later, Byun \cite[Theorem 1]{byun} proved that for this type of manifold, the Stiefel--Whitney numbers vanish under the additional assumption that the restriction of the mod $2$ Hurewicz homomorphism to the center of $\pi_1(M)$ is nonzero. More recently, Hebestreit, Land, L\"uck, and Randal-Williams \cite[Corollary]{HLLRW} showed, under extra hypotheses involving the block Borel and Burghelea conjectures, that all Pontryagin numbers of such manifolds vanish, and asked whether they must be (unoriented) nullbordant \cite[Question 6]{HLLRW}.

In this note we show the existence of closed oriented smooth aspherical manifolds whose fundamental groups have nontrivial center and for which some Stiefel--Whitney number is nonzero. Hence the answer to the question of Hebestreit, Land, L\"uck, and Randal-Williams is no. In turn, their general conjecture, formulated in \cite{HLLRW}, about the vanishing of the tautological classes in $H^*(B\widetilde{\Top}_h(M);R)$ for all commutative rings $R$, is false. Furthermore, we show that the cohomology of $B\Diff_h(M)$ and hence that of $B\widetilde{\Top}_h(M)$ is nontrivial in arbitrarily high degrees.

We now state our results precisely. For a closed oriented smooth $n$-manifold $N$, let $\Diff_h(N)$ denote the topological group of diffeomorphisms of $N$ which are homotopic to the identity. Let
\beq
\pi:E\Diff_h(N)\times_{\Diff_h(N)}N\to B\Diff_h(N)
\eeq
be the universal smooth $N$-bundle over the classifying space $B\Diff_h(N)$. We denote its vertical tangent bundle by $\msf{T}_v(\pi)$. If $\mathfrak{c}$ is a polynomial in the Stiefel--Whitney classes, its corresponding tautological class is defined by:
\beq
\kappa_\mathfrak{c}:=\pi_!\bigl(\mathfrak{c}(\msf{T}_v(\pi))\bigr)
\in H^*(B\Diff_h(N);\F_2),
\eeq
where $\pi_!$ is integration along the fibers, or the Gysin homomorphism.

We denote by $v=1+v_1+v_2+\cdots\in H^*(B\On;\F_2)$ the universal total Wu class. It is characterized by the identity $\Sq(v)=w$, where $w=1+w_1+w_2+\cdots\in H^*(B\On;\F_2)$ is the total Stiefel--Whitney class and $\Sq$ is the total Steenrod square.
\begin{theoremalpha}\label{thm:tautological}
For every even integer $m\geq 2$, there exists a $(2m+1)$-dimensional closed, oriented, smooth aspherical manifold $M$ whose fundamental group has nontrivial center and for which
\beq
\kappa_{v_m^{2r+1}\Sq^1v_m}\neq 0\in H^{2mr}(B\Diff_h(M);\F_2)
\eeq
for infinitely many $r\geq 0$. Moreover, for each such $r$, the class $\kappa_{v_m^{2r+1}\Sq^1v_m}$ is not nilpotent.
\end{theoremalpha}
For a connected space $X$, let $\mathrm Z(\pi_1(X))$ denote the center of its fundamental group. 
\begin{theoremalpha}\label{thm:bordism}
For every $a\geq 1$ and $b\geq 0$ there is a closed, oriented, smooth, aspherical manifold $M_{a,b}$ of dimension $5a+4b$ such that $\mathrm Z(\pi_1(M_{a,b}))\supset\Z^a$
and $M_{a,b}$ is not nullbordant.
\end{theoremalpha}
Note that in this way we obtain examples in dimensions $5,9,10,13,14,15$ and every dimension greater than or equal to $17$.

The idea that led to the proof of Theorem \ref{thm:bordism} and ultimately to that of Theorem \ref{thm:tautological} is the following (compare with the construction in \cite{cwy}): we apply Davis--Januszkiewicz hyperbolization \cite{dj} to complex conjugation on $\C P^2$. The resulting map is an involution with a fixed point on a certain closed aspherical manifold $V$. This fixed point and the finite order of the automorphism will guarantee a nontrivial center on the fundamental group of the mapping torus $M$ of this involution. Either by direct calculation, or by comparing $M$ with a Dold manifold, we can show that $M$ has a nonzero Stiefel--Whitney number, settling the result in dimension 5. Taking products of $M$ with itself and with fake projective planes (or any other closed orientable non-nullbordant aspherical $4$-manifold) yields the result in the other dimensions.
To obtain nontrivial tautological classes of higher degree, we exploit the fact that since the monodromy of the mapping torus $M$ has order $2$ then it comes equipped with an action of the cyclic group $C_2$ of order $2$ (in fact, of the circle $S^1$), so tautological classes can be detected in the cohomology of $\R P^{\infty}$.
\subsection*{Acknowledgments} 
I am grateful to Oscar Randal-Williams for his comments on a previous version of this paper, and especially for pointing me to \cite{lmp}, where the characteristic number $\left\langle v_m\Sq^1v_m(M),[M]\right\rangle$ that I ultimately use here is studied. He also suggested that I use the technique of equivariant localization to make the computations in Section \ref{sec:computations}. I thank Eduardo Reyes for useful comments and for helping me with the proof of Lemma \ref{lem:Lucas}. I also benefited from a conversation with Fabian Hebestreit. This work was supported by ANID Chile through Fondecyt Regular Grant 1250727.
\section{Functoriality of the hyperbolization procedure}\label{sec:hyperbolization}
We recall that the Davis--Januszkiewicz hyperbolization \cite{dj} (or Gromov's cylinder construction \cite[Section 3.4]{gromov}) assigns to a finite simplicial complex $K$ a compact nonpositively curved cubical complex $h(K)$ (which is therefore aspherical). This assignment is functorial for simplicial embeddings, and also produces a continuous map $h_K:h(K)\to K$ which is natural and surjective on homology with any coefficients. If $K$ is a smooth triangulation of a smooth manifold, then $h(K)$ is a smooth manifold and the map $h_K$ is stably tangential, namely
\beq
Th(K)\oplus\epsilon^r\cong h_K^*TK\oplus\epsilon^r,
\eeq
for some $r\geq 0$.  These properties are proved in \cite[Lemmas 1b.3, 1b.4 and 1d.1, Proposition 1f.5, Corollary 1f.6, Theorem 4a.3 and Section 4c]{dj}.

Fix a degree 1 hyperbolized $n$-simplex $(X^n,f)$ as constructed in \cite[Section 4c]{dj}. If $K$ is an $n$-dimensional simplicial complex and $K'$ is its barycentric subdivision, let $d_K:|K'|\to\sigma^n$ be the simplicial map which sends the barycenter of a simplex to the vertex indexed by its dimension. The hyperbolization of $K$ is the fiber product
\beq
h(K)=X^n\triangle K'=\{(x,y)\in X^n\times |K'|:f(x)=d_K(y)\},
\eeq
and $h_K$ is the restriction of the second projection. Also recall that this $n$-simplex $X^n$ has trivial tangent bundle, so we fix a trivialization $TX\cong\epsilon^n$.
\begin{lemma}\label{lem:natural-tangent}
If $K$ is a smooth triangulation of a closed smooth $n$-manifold, then the Davis--Januszkiewicz hyperbolization gives a short exact sequence
\beq
0\to Th(K)\to p_X^*TX^n\oplus h_K^*TK\xto{D\Phi}\epsilon^n\to 0,
\eeq
which is natural with respect to smooth simplicial isomorphisms. Moreover, if a finite group $G$ acts smoothly and simplicially on $K$, then the stable tangent bundle isomorphism
\beq
Th(K)\oplus\epsilon^n\cong h_K^*TK\oplus\epsilon^n
\eeq
may be chosen $G$-equivariantly.
\end{lemma}
\begin{proof}
Identify the affine span of $\sigma^n$ with $\R^n$ and consider the map $\Phi_K:X^n\times|K'|\to\R^n$ given by $\Phi_K(x,y)=f(x)-d_K(y)$. By the proof of \cite[Proposition 1f.5]{dj}, this map can be regarded as smooth and transverse to $0$, and $h(K)=\Phi_K^{-1}(0)$. Its differential therefore gives the stated short exact sequence.

If $g:K\to L$ is a smooth simplicial isomorphism, then the induced map $g':K'\to L'$ preserves the dimension of each simplex, and hence $d_L\circ g'=d_K$. Thus $\Phi_L\circ(\id_{X^n}\times g')=\Phi_K$, so the exact sequences are natural.

Suppose now that a finite group $G$ acts smoothly and simplicially on $K$. By naturality, the exact sequence is $G$-equivariant. Choose a splitting $\sigma_K:\epsilon^n\to p_X^*TX^n\oplus h_K^*TK$ of $D\Phi_K$ and average it over $G$. The resulting splitting is $G$-equivariant and still satisfies $D\Phi_K\circ\sigma_K=\id$. It therefore gives a $G$-equivariant isomorphism
\beq
Th(K)\oplus\epsilon^n\cong p_X^*TX^n\oplus h_K^*TK.
\eeq
Finally, $G$ acts trivially on the $X^n$-coordinate, so the fixed trivialization $p_X^*TX^n\cong\epsilon^n$ is $G$-equivariant. The result follows.
\end{proof}
\section{Some aspherical manifolds with non-trivial center}\label{sec:construction}
Let $N$ be a closed, oriented, connected, smooth $n$-manifold and $\phi:N\to N$ an orientation preserving diffeomorphism of order $k$ with a fixed point. By Illman's equivariant triangulation theorem \cite[Theorem]{illman}, there is a finite smooth triangulation $K$ of $N$ on which $\phi$ acts simplicially. After barycentric subdivision we may suppose that $\phi$ fixes a vertex $v\in K$.

Put $W:=h(K)$, the hyperbolization of $K$, and $\rho:=h(\phi):W\to W$. By functoriality and naturality (see Section \ref{sec:hyperbolization}), we have that $\rho^k=\id$, $h_K\circ\rho=\phi\circ h_K$, and $h(v)$ is a fixed point of $W$.
Furthermore the stable tangent bundle formula gives
\beq
w_1(W)=h_K^*w_1(N)=0,
\eeq
so $W$ is orientable. Since $h_K$ has degree $1$, we may orient $W$ so that $\deg(h_K)=1$. Taking degrees in the identity $h_K\circ\rho=\phi\circ h_K$ gives $\deg(\rho)=\deg(\phi)=1$. Therefore $\rho$ preserves orientation. It is also clear that $\rho$ has order $k$. In fact, functoriality implies that $\rho^k=\id$ and, combined with surjectivity of $h_K$, gives that a smaller power of $\rho$ cannot be the identity, because otherwise $\phi$ would have order less than $k$.

Denote by $T(\rho)$ the mapping torus of $\rho:W\to W$. This is the space $W\times\R/\left((x,t)\sim(\rho(x),t+1)\right)$. It is a closed, connected, oriented, smooth $(n+1)$-manifold.
\begin{lemma}\label{lem:aspherical-center}
The manifold $T(\rho)$ is aspherical and $\mathrm Z(\pi_1(T(\rho)))$ is nontrivial.
\end{lemma}
\begin{proof}
Asphericity follows directly from the fact that the mapping-torus projection $T(\rho)\to S^1$ is a fiber bundle with fiber $W$, and both $S^1$ and $W$ are aspherical. For the second part of the statement, choose a fixed point $y_0\in W$ of $\rho$. Using $y_0$ as basepoint, Seifert--van Kampen's theorem gives
\beq
\pi_1(T(\rho))\cong\pi_1(W)\rtimes_{\rho_*}\langle t\rangle.
\eeq
where $t$ is the stable letter represented by the loop going once around the base circle. Since $\rho(y_0)=y_0$ and $\rho^k=\id$, the induced automorphism satisfies $\rho_*^k=\id$. It follows that
\beq
t^kgt^{-k}=\rho_*^k(g)=g
\eeq
for all $g\in\pi_1(W)$. The nontrivial element $t^k$ also commutes with $t$, so it lies in $\mathrm Z(\pi_1(T(\rho)))$. 
\end{proof}
\subsection{Finite cyclic group actions and hyperbolization}
Recall that $T(\rho)$ is the mapping torus of $\rho:W\to W$. Note that since $\rho$ has order $k$, the cyclic group $C_k$ of order $k$ acts on $T(\rho)$ by
\beq
\widehat\rho([x,t])=[\rho(x),t].
\eeq
This map is isotopic to the identity and so it defines a homomorphism $C_k\to\Diff_h(T(\rho))$ which induces a map $\beta:BC_k\to B\Diff_h(T(\rho))$ between classifying spaces.

Let
\beq
p_\rho:EC_k\times_{C_k}T(\rho)\to BC_k
\eeq
be the associated smooth $T(\rho)$-bundle. Hence, for every polynomial $\mathfrak{c}\in H^*(B\On;\F_2)$ in the universal Stiefel--Whitney classes, we have
\beq
\beta^*\kappa_{\mathfrak{c}}=(p_\rho)_!\left(\mathfrak{c}(\msf{T}_v(p_\rho))\right).
\eeq
We now use properties of the hyperbolization to reduce the calculation of the right-hand side to the corresponding bundle with fiber $T(\phi)$, where, as before, the map
$\phi:N\to N$ is orientation preserving of order $k$. Note that there is also an orientation preserving, order $k$ diffeomorphism $\widehat \phi:T(\phi)\to T(\phi)$, sending $[q,t]\in T(\phi)$ to $[\phi(q),t]$. Because of the identity $h_K\circ\rho=\phi\circ h_K$, the map $\varphi:T(\rho)\to T(\phi)$ given by $\varphi([x,t])=[h_K(x),t]$ is equivariant with respect to these $C_k$-actions. Applying the Borel construction therefore gives a map
\beq
\overline\varphi:EC_k\times_{C_k}T(\rho)\to EC_k\times_{C_k}T(\phi)
\eeq
over $BC_k$. Let $p_\phi:EC_k\times_{C_k}T(\phi)\to BC_k$ denote the corresponding $T(\phi)$-bundle.
\begin{lemma}\label{lem:sw-Dold}
The map $\overline\varphi$ has degree one on every fiber. Moreover, the Stiefel--Whitney classes of $\msf{T}_v(p_{\rho})$ and $\msf{T}_v(p_\phi)$ satisfy
\beq
w_i(\msf{T}_v(p_\rho))=\overline\varphi^*w_i(\msf{T}_v(p_\phi))
\eeq
for all $i$.
\end{lemma}
\begin{proof}
The map $\overline\varphi$ restricts on every fiber to $\varphi:T(\rho)\to T(\phi)$. This map has degree one because it is a map of mapping tori over $S^1$ whose restriction to each fiber is the degree-one map $h_K:W\to\N$.

By Lemma \ref{lem:natural-tangent}, the stable tangent bundle isomorphism for $h_K$ may be chosen equivariantly with respect to $\rho$ and $\phi$. Taking mapping tori and adding the tangent line in the $S^1$-direction gives a stable tangent bundle isomorphism for $\varphi$ which is equivariant with respect to $\widehat\rho$ and $\widehat \phi$. After applying the Borel construction, this gives an isomorphism of stable vertical tangent bundles 
\beq
\msf{T}_v(p_\rho)\oplus\epsilon^s
\cong
\overline\varphi^*\msf{T}_v(p_\phi)\oplus\epsilon^s
\eeq
for some $s\geq0$. The result follows by stability and naturality of the Stiefel--Whitney classes.
\end{proof}
Since $\overline\varphi$ has degree one on every fiber, integration along the fiber satisfies
\beq
(p_\rho)_!\bigl(\overline\varphi^*z\bigr)=(p_\phi)_!(z)
\eeq
for every $z\in H^*(EC_k\times_{C_k}T(\phi);\F_2)$. Thus we obtain the following corollary.
\begin{corollary}\label{cor:beta}
Let $\mathfrak{c}$ be a polynomial in the Stiefel--Whitney classes. Then
\beq
\beta^*\kappa_{\mathfrak{c}}=(p_\phi)_!\left(\mathfrak{c}(\msf{T}_v(p_\phi))\right).
\eeq
\end{corollary}
Thus the problem of computing the tautological classes for $p_\rho:EC_k\times_{C_k}T(\rho)\to BC_k$ has been reduced to calculating the tautological classes of the bundle
$p_\phi:EC_k\times_{C_k}T(\phi)\to BC_k$.
\subsection{Tautological classes and equivariant localization}\label{sec:localization}
We now restrict to the case when $k=2$. To compute the tautological classes of $p_\phi:EC_2\times_{C_2}T(\phi)\to BC_2$, we use the technique called ``equivariant localization''. This technique is presented in \cite{allday-puppe} and we refer the reader there for details and further references. Chapters 3 and 5 of that book are especially useful here. We recall the main facts. 

For a closed oriented smooth manifold $X$ with an orientation preserving action of $C_2$, write $p_X:X//C_2\to BC_2$, where $X//C_2:=EC_2\times_{C_2}X$ and $H^*_{C_2}(X):=H^*(X//C_2;\F_2)$.
In particular,
\beq
H^*_{C_2}(\pt)=H^*(BC_2;\F_2)\cong\F_2[t],
\eeq
with $|t|=1$.

Let $F_1,\ldots,F_\ell$ be the connected components of the fixed point set $X^{C_2}$, and let
$i_j:F_j//C_2\to X//C_2$ denote the corresponding inclusions. Since $C_2$ acts trivially on $F_j$, one has
\beq
F_j//C_2\cong F_j\times BC_2.
\eeq
Let $\pi_j:F_j\times BC_2\to BC_2$ be the projection. The fiber-integration map
\beq
(\pi_j)_!:H^*_{C_2}(F_j)\to H^{*-\dim(F_j)}(BC_2;\F_2)
\eeq
is therefore given by slant product with the mod $2$ fundamental class of $F_j$.

Let $\nu_j\to F_j$ be the normal bundle of $F_j\subset X$, equipped with the $C_2$-action. Its Borel construction is the vector bundle
\beq
\nu_j^{C_2}:=\nu_j//C_2\to F_j//C_2=F_j\times BC_2.
\eeq
If $r_j$ denotes the rank of $\nu_j$, its \textit{equivariant} mod $p$ \textit{Euler}
class is defined as
\beq
e_{C_2}(\nu_j):=w_{r_j}(\nu_j^{C_2})\in H^{r_j}_{C_2}(F_j).
\eeq
Let $S=\{1,t,t^2,\ldots\}\subset\F_2[t]$. The localization theorem in equivariant cohomology, as recalled in \cite[Theorem 3.1.6]{allday-puppe} says that the restriction map
\beq
\bigoplus_{j=1}^{\ell}i_j^*:S^{-1}H^*_{C_2}(X)\to\bigoplus_{j=1}^{\ell}S^{-1}H^*_{C_2}(F_j)
\eeq
is an isomorphism. Moreover, each class $e_{C_2}(\nu_j)$ becomes a unit after localization. Using the formula $i_j^*(i_j)_!(x)=e_{C_2}(\nu_j)x$, for every $x\in H^*_{C_2}(F_j)$, one obtains the following localization formula (see \cite[Section 5.3]{allday-puppe}): if $z\in H^*_{C_2}(X)$, then, after inverting $t$, we have
\beq
z=\sum_{j=1}^{\ell}(i_j)_!\left(\frac{i_j^*z}{e_{C_2}(\nu_j)}\right).
\eeq
Since $p_X\circ i_j=\pi_j$, functoriality of the Gysin maps gives
\begin{equation}\label{eq:loc-formula}
(p_X)_!(z)=\sum_{j=1}^{\ell}(\pi_j)_!\left(\frac{i_j^*z}{e_{C_2}(\nu_j)}\right)\in S^{-1}H^*(BC_2;\F_2)\cong\F_2[t,t^{-1}].
\end{equation} 
\section{Mapping torus of the hyperbolization of complex conjugation}\label{sec:computations}
The results of Section \ref{sec:construction} reduce the proofs of Theorems \ref{thm:tautological} and \ref{thm:bordism} to finding mapping tori with finite-order orientation preserving monodromy and with plenty of characteristic classes. Dold manifolds provide such examples. They were introduced by Dold in \cite{dold} to construct explicit generators of the unoriented bordism ring. Dold manifolds, typically denoted by $P(n,m)$, are quotients of the form $S^n\times_{C_2}\C P^m$ where $C_2$ acts on the $n$-sphere $S^n$ by the antipodal map and on the complex projective space $\C P^m$ by complex conjugation. When $n=1$, the Dold manifold $P(1,m)$ will be denoted by $P_m$. It is the mapping torus of complex conjugation $c:[z_0:\cdots:z_m]\mapsto[\overline{z_0}:\cdots:\overline{z_m}]$ which is an order $2$ diffeomorphism of $\C P^m$. As discussed in Section \ref{sec:construction}, the manifold $P_m$ inherits an action of $C_2$ and we denote by
\beq
p_D:EC_2\times_{C_2}P_m\to BC_2
\eeq
the corresponding Dold manifold bundle. 
\subsection{The fixed point set $P_m^{C_2}$}
To apply the technique of equivariant localization as just described, it remains to identify the fixed point set of the $C_2$-action on the Dold manifold $P_m$, its equivariant normal bundle, and the restriction of $\msf{T}_v(p_D)$ to the fixed point set. Using the identification $P_m=(S^1\times\C P^m)/((v,q)\sim(-v,c(q)))$, we can identify the fixed point set as
\beq
P_m^{C_2}=\bigl(\R P^m\times S^1\bigr)/\bigl((\ell,v)\sim(\ell,-v)\bigr)\cong\R P^m\times S^1.
\eeq
Let $\nu(P_m^{C_2})\to P_m^{C_2}$ be the normal bundle of the inclusion $P_m^{C_2}\subset P_m$, equipped with the action induced by the $C_2$-action on $P_m$. Denote by $\mu\to S^1$ the nontrivial line bundle over the circle. Also denote by $\operatorname{pr}_{\R P^m}$ and $\operatorname{pr}_{S^1}$ the projections of $\R P^m\times S^1$ onto each of the factors.
\begin{lemma}\label{lem:normal-fixedpts}
There is an isomorphism of vector bundles over $P_m^{C_2}=\R P^m\times S^1$
\beq
\nu(P_m^{C_2})\cong\operatorname{pr}_{\R P^m}^*T\R P^m\otimes\operatorname{pr}_{S^1}^*\mu.
\eeq
\end{lemma}
\begin{proof}
First recall that multiplication by the imaginary unit $i$ gives an isomorphism $J:T\R P^m\to\nu(\R P^m)$, where $\nu(\R P^m)$ is the normal bundle of the canonical inclusion $\R P^m\subset\C P^m$. Since $c$ restricts to the identity on $\R P^m$ and complex conjugation anticommutes with multiplication by $i$, we have that the derivative of $c$ at $\ell\in\R P^m$ satisfies $Dc_\ell(J_\ell(A))=-J_\ell(A)$ for every $A\in T_\ell\R P^m$. It follows that
\beq
\nu(P_m^{C_2})=\bigl(S^1\times\nu(\R P^m)\bigr)/\bigl((v,J_\ell(A))\sim(-v,-J_\ell(A))\bigr).
\eeq
Therefore the map
\beq
\Phi:\nu(P_m^{C_2})\to\operatorname{pr}_{\R P^m}^*T\R P^m\otimes\operatorname{pr}_{S^1}^*\mu
\eeq
given by $\Phi\bigl([v,J_\ell(A)]\bigr)=A\otimes[v,1]$ is a well-defined isomorphism of vector bundles.
\end{proof}
Let now $\lambda\to BC_2$ be the line bundle induced by the sign representation $C_2\to\GL_1(\R)$. 
\begin{lemma}\label{lem:C2-normal-fixedpts}
There is an isomorphism of vector bundles over $P_m^{C_2}\times BC_2=\R P^m\times S^1\times BC_2$
\beq
\nu^{C_2}(P_m^{C_2}):=EC_2\times_{C_2}\nu(P_m^{C_2})\cong\operatorname{pr}_{\R P^m}^*T\R P^m\otimes\operatorname{pr}_{S^1}^*\mu\otimes\operatorname{pr}_{BC_2}^*\lambda.
\eeq
\end{lemma}
\begin{proof}
Let $g$ denote the nontrivial element of $C_2$. Since the action of $C_2$ on $P_m$ is induced by complex conjugation on $\C P^m$, its action on the normal bundle of $P_m^{C_2}$ is given by 
\beq
g\cdot[v,J_\ell(A)]=[v,-J_\ell(A)]=[v,J_\ell(-A)].
\eeq
Under the isomorphism $\Phi$ of Lemma \ref{lem:normal-fixedpts}, this gives 
\beq
\Phi\bigl(g\cdot[v,J_\ell(A)]\bigr)=-A\otimes[v,1]=-\Phi\bigl([v,J_\ell(A)]\bigr).
\eeq
Thus $C_2$ acts by multiplication by $-1$ on every normal direction. Define
\beq
\Psi:EC_2\times_{C_2}\nu(P_m^{C_2})\to\operatorname{pr}_{\R P^m}^*T\R P^m\otimes\operatorname{pr}_{S^1}^*\mu\otimes\operatorname{pr}_{BC_2}^*\lambda
\eeq
by $\Psi\bigl([e,[v,J_\ell(A)]]\bigr)=A\otimes[v,1]\otimes[e,1]$.
A direct check shows that this is a well-defined isomorphism of vector bundles.
\end{proof}
\subsection{The equivariant Euler class}
We now turn to the calculation of the equivariant Euler class. Let us fix generators $u\in H^1(\R P^m)$, $s\in H^1(S^1)$, $t=w_1(\lambda)\in H^1(BC_2)$. We use the same notation for their pullbacks to $\R P^m\times S^1\times BC_2$. Thus there is a ring isomorphism
\beq
H^*(P_m^{C_2}\times BC_2;\F_2)\cong\F_2[u,s,t]/(u^{m+1},s^2).
\eeq
\begin{lemma}\label{lem:euler-normal-fixedpts}
The equivariant Euler class of the normal bundle of $P_m^{C_2}\subset P_m$ is
\beq
e_{C_2}(\nu(P_m^{C_2}))=\sum_{j=0}^m\binom{m+1}{j}u^j(s+t)^{m-j},
\eeq
where the binomial coefficients are reduced modulo $2$.
\end{lemma}
\begin{proof}
By definition, the equivariant mod $2$ Euler class of $\nu(P_m^{C_2})$ is the mod $2$ Euler class of $\nu^{C_2}(P_m^{C_2})$. Since this bundle has rank $m$, this is its $m$th Stiefel--Whitney class. By Lemmas \ref{lem:normal-fixedpts} and \ref{lem:C2-normal-fixedpts}, it is obtained by tensoring $\operatorname{pr}_{\R P^m}^*T\R P^m$ with the line bundle $\operatorname{pr}_{S^1}^*\mu\otimes\operatorname{pr}_{BC_2}^*\lambda$ whose first Stiefel--Whitney class is $s+t$. Hence
\beq
e_{C_2}(\nu(P_m^{C_2}))=\sum_{j=0}^m w_j(T\R P^m)(s+t)^{m-j}.
\eeq
But $w_j(T\R P^m)=\binom{m+1}{j}u^j$, which gives the claimed expression.
\end{proof}
\begin{remark}
Note that $e_{C_2}(\nu(P_m^{C_2}))\equiv t^m$ modulo the nilpotent ideal $(u,s)$. Hence, after inverting $t$, this Euler class is a unit in $H^*(P_m^{C_2};\F_2)[t,t^{-1}]$, as expected.
\end{remark}
Let $i:P_m^{C_2}\times BC_2=P_m^{C_2}//C_2\to P_m//C_2$ be the inclusion of the fixed point set, and put $\theta=s+t$ and
\beq
e:=e_{C_2}(\nu(P_m^{C_2}))=\sum_{j=0}^m\binom{m+1}{j}u^j\theta^{m-j}.
\eeq
For the next lemma we need the following setup. Let $\xi\to B$ be a real vector bundle over a topological space $B$, and let $v(\xi)=1+v_1(\xi)+v_2(\xi)+\cdots$ be its total Wu class. Let $\alpha,\beta\in H^1(B;\F_2)$, and assume that there is a natural number $m$ such that $\alpha^{m+1}=0$. Furthermore, suppose that $\xi$ satisfies
\beq
w(\xi)(1+\beta)=(1+\alpha)^{m+1}(1+\alpha+\beta)^{m+1}.
\eeq
\begin{lemma}\label{lem:Wu-class}
The $m$-th Wu class of $\xi$ is
\beq
v_m(\xi)=\sum_{j=0}^m\binom{m+1}{j}\alpha^j\beta^{m-j}.
\eeq
\end{lemma}
\begin{proof}
By the Wu formula, $w(\xi)=\operatorname{Sq}(v(\xi))$, which can be formally written as
\beq
v(\xi)=\operatorname{Sq}^{-1}(w(\xi)).
\eeq
For a degree $1$ class $x$, we have $\operatorname{Sq}(x)=x+x^2$, so the inverse total Steenrod square replaces $x$ by $x+x^2+x^4+x^8+\cdots$. Set $U=\alpha+\alpha^2+\alpha^4+\alpha^8+\cdots$ and $\Theta=\beta+\beta^2+\beta^4+\beta^8+\cdots$. Then
\beq
v(\xi)=(1+U)^{m+1}(1+U+\Theta)^{m+1}(1+\Theta)^{-1}.
\eeq
Since $U+U^2=\alpha$, we have $(1+U)(1+U+\Theta)=\alpha+(1+U)(1+\Theta)$. Consequently, using $\alpha^{m+1}=0$, we obtain
\beq
v(\xi)=\sum_{j=0}^m\binom{m+1}{j}\alpha^j(1+U)^{m-j+1}(1+\Theta)^{m-j}.
\eeq
We claim that, for every $q\geq 0$, the degree $q$ part of $(1+U)^{q+1}(1+\Theta)^q$ is $\beta^q$. We prove this by induction on $q$. The claim is clear for $q=0$. Suppose first that $q=2k$, $k\geq 1$. Since $1+U=(1+U^2)+\alpha$, we have
\beq
(1+U)^{2k+1}(1+\Theta)^{2k}=(1+U^2)^{k+1}(1+\Theta^2)^k+\alpha(1+U^2)^k(1+\Theta^2)^k.
\eeq
Note that the summand $\alpha(1+U^2)^k(1+\Theta^2)^k$ has only terms of odd degree. Hence, the degree $2k$ part of $(1+U)^{2k+1}(1+\Theta)^{2k}$ is the degree $2k$ part of $(1+U^2)^{k+1}(1+\Theta^2)^k$.
On the other hand
\beq
(1+U^2)^{k+1}(1+\Theta^2)^k=\left((1+U)^{k+1}(1+\Theta)^k\right)^2.
\eeq
Therefore the degree $2k$ part is the square of the degree $k$ part of $(1+U)^{k+1}(1+\Theta)^k$. By the induction hypothesis, this is $(\beta^k)^2=\beta^{2k}$.

The case $q=2k+1$ is similar: $(1+U)^{2k+2}(1+\Theta)^{2k+1}=(1+\Theta)\left((1+U)^{k+1}(1+\Theta)^k\right)^2$. Since the square has only terms of even degree and $\beta$ is the only odd-degree term of $1+\Theta$, the induction hypothesis implies that the degree $2k+1$ part is $\beta(\beta^k)^2=\beta^{2k+1}$. This proves the claim.

Taking now $q=m-j$, we conclude that the degree $m-j$ part of $(1+U)^{m-j+1}(1+\Theta)^{m-j}$ is $\beta^{m-j}$. Comparing degree $m$ in the formula for $v(\xi)$ therefore gives
\beq
v_m(\xi)=\sum_{j=0}^m\binom{m+1}{j}\alpha^j\beta^{m-j}.
\eeq
\end{proof}
\begin{lemma}\label{lem:vertical-tangent-fixedpts}
Let $i:EC_2\times_{C_2}P_m^{C_2}=BC_2\times P_m^{C_2}\hookrightarrow EC_2\times_{C_2}P_m$ be the inclusion. Then there is an isomorphism
\beq
i^*\msf{T}_v(p_D)\cong\operatorname{pr}_{P_m^{C_2}}^*TP_m^{C_2}\oplus\nu^{C_2}(P_m^{C_2}).
\eeq
Moreover, $i^*v_m(\msf{T}_v(p_D))=e$ and $i^*\Sq^1v_m(\msf{T}_v(p_D))=
\begin{cases}
ue\ \ \text{if}\ \ m\ \ \text{is even}\\
\theta e\ \ \text{if}\ \ m\ \ \text{is odd}
\end{cases}$.
\end{lemma}
\begin{proof}
Choose a $C_2$-invariant Riemannian metric on $P_m$. Along the fixed point set, it gives a $C_2$-equivariant splitting $TP_m|_{P_m^{C_2}}\cong TP_m^{C_2}\oplus\nu(P_m^{C_2})$. Since $C_2$ acts trivially on $P_m^{C_2}$, applying the Borel construction gives
\beq
i^*\msf{T}_v(p_D)\cong\operatorname{pr}_{P_m^{C_2}}^*TP_m^{C_2}\oplus\nu^{C_2}(P_m^{C_2}).
\eeq
The Whitney product formula gives
\beq
w\bigl(i^*\msf{T}_v(p_D)\bigr)=w(TP_m^{C_2})w(\nu^{C_2}(P_m^{C_2})).
\eeq
Since $P_m^{C_2}=\R P^m\times S^1$, we get $w(TP_m^{C_2})=(1+u)^{m+1}$. To compute the second factor, we add the line bundle $\operatorname{pr}_{S^1}^*\mu\otimes\operatorname{pr}_{BC_2}^*\lambda$ to both sides of the isomorphism of Lemma \ref{lem:C2-normal-fixedpts} to obtain
\beq
\nu^{C_2}(P_m^{C_2})\oplus (\operatorname{pr}_{S^1}^*\mu\otimes\operatorname{pr}_{BC_2}^*\lambda)\cong (T\R P^m\oplus\epsilon^1)\otimes\operatorname{pr}_{S^1}^*\mu\otimes\operatorname{pr}_{BC_2}^*\lambda.
\eeq
Using the fact that $T\R P^m\oplus\epsilon^1$ splits as a sum of $m+1$ copies of the canonical line bundle, whose first Stiefel--Whitney class is $u$, the Whitney formula gives
\beq
w(\nu^{C_2}(P_m^{C_2}))(1+\theta)=(1+u+\theta)^{m+1}.
\eeq
Therefore
\beq
w\bigl(i^*\msf{T}_v(p_D)\bigr)(1+\theta)=(1+u)^{m+1}(1+u+\theta)^{m+1}.
\eeq
Applying Lemma \ref{lem:Wu-class} we obtain
\beq
i^*v_m(\msf{T}_v(p_D))=\sum_{j=0}^m\binom{m+1}{j}u^j\theta^{m-j}=e.
\eeq
Finally, naturality and the Wu formula give $i^*\Sq^1v_m=\Sq^1e=w_1(\nu^{C_2}(P_m^{C_2}))e$.
But $w_1(\nu^{C_2}(P_m^{C_2}))=(m+1)u+m\theta=u$, hence $i^*\Sq^1v_m=ue$ if $m$ is even and $i^*\Sq^1v_m=\theta e$ if $m$ is odd.
\end{proof}
For every $r\geq0$, set
\beq
z_r:=v_m(\msf{T}_v(p_D))^{2r+1}\Sq^1v_m(\msf{T}_v(p_D))\in H^{2m(r+1)+1}_{C_2}(P_m).
\eeq
\begin{lemma}\label{lem:Dold-tautological-class}
For every positive even integer $m>1$ and every $r\geq 0$, the following equality holds
\beq
(p_D)_!(z_r)=\binom{(m+1)(2r+1)}{m-1}t^{2mr}\in H^{2mr}(BC_2;\F_2).
\eeq
\end{lemma}
\begin{proof}
Let $q:P_m^{C_2}\times BC_2\to BC_2$ be the projection. Lemma \ref{lem:vertical-tangent-fixedpts} and the localization formula \eqref{eq:loc-formula} give
\beq
(p_D)_!(z_r)=q_!\left(\frac{i^*z_r}{e}\right)=q_!\left(ue^{2r+1}\right)\in\F_2[t,t^{-1}].
\eeq
Since $u^{m+1}=0$, Lemmas \ref{lem:Wu-class} and \ref{lem:vertical-tangent-fixedpts} give the identity $\theta e=(u+\theta)^{m+1}$. After inverting $t$, the class $\theta=s+t$ is a unit, and hence $ue^{2r+1}=u\theta^{-(2r+1)}(u+\theta)^{(m+1)(2r+1)}$.
The coefficient of $u^m$ in this expression is
\beq
\binom{(m+1)(2r+1)}{m-1}\theta^{2mr+1}.
\eeq
Since $s^2=0$, one has $\theta^{2mr+1}=t^{2mr+1}+st^{2mr}$. Integration along the fiber $P_m^{C_2}=\R P^m\times S^1$ of the trivial bundle $q$ extracts the coefficient of $u^ms$, and therefore
\beq
q_!(ue^{2r+1})=\binom{(m+1)(2r+1)}{m-1}t^{2mr}.
\eeq
Since $\F_2[t]\to\F_2[t,t^{-1}]$ is injective, the same equality holds in $H^*(BC_2;\F_2)$.
\end{proof}
\begin{lemma}\label{lem:Lucas}
For every positive even integer $m$, the set
\beq
\msf{Q}_m:=\left\{r\geq 0:\binom{(m+1)(2r+1)}{m-1}=1\in\F_2\right\}
\eeq
is infinite.
\end{lemma}
\begin{proof}
Note that since $m+1$ is odd, Euler's theorem gives $2^{\varphi(m+1)}\equiv 1\pmod{m+1}$. For every $k\geq 1$, let $s_k=k\varphi(m+1)$. Then $m+1$ divides $2^{s_k}-1$ and the quotient is an odd number, which we write as $2r_k+1$. Therefore $(m+1)(2r_k+1)=2^{s_k}-1$. Note that the binary expansion of $2^{s_k}-1$ consists of consecutive $1$'s. So if $k\geq 1$, then Lucas' theorem gives $\binom{(m+1)(2r_k+1)}{m-1}=\binom{2^{s_k}-1}{m-1}=1\in\F_2$. Thus $r_k\in\msf{Q}_m$ for every $k\geq 1$. Since the sequence $s_k$ is strictly increasing, so is the sequence $r_k$, and therefore $\msf{Q}_m$ is infinite.
\end{proof}
\section{Proofs of the main theorems}
We can now complete the proof of Theorems \ref{thm:tautological} and \ref{thm:bordism}.
\subsection{Proof of Theorem \ref{thm:tautological}}
Let $M$ be the mapping torus of the Davis--Januszkiewicz hyperbolization of complex conjugation on $\C P^m$. For even $m$, complex conjugation is an orientation preserving involution with fixed points. Therefore, by the discussion in Section \ref{sec:construction} and Lemma \ref{lem:aspherical-center}, $M$ is a closed oriented aspherical manifold of dimension $2m+1$ whose fundamental group has nontrivial center, and is equipped with a (homotopically trivial) action of $C_2$, inducing the map $\beta:BC_2\to B\Diff_h(M)$. By Corollary \ref{cor:beta}, we have that $\beta^*(\kappa_{\mathfrak{c}})=(p_D)_!(\mathfrak{c}(\msf{T}_v(p_D)))$ for any characteristic class $\mathfrak{c}\in H^*(B\On;\F_2)$, where 
$p_D:EC_2\times_{C_2}P_m\to BC_2$ is the fiber bundle projection and, as above, $P_m$ is the Dold manifold (the mapping torus of complex conjugation).

Now by Lemma \ref{lem:Dold-tautological-class} and Lemma \ref{lem:Lucas}, the set $\msf{Q}_m$ is infinite and, for every $r\in\msf{Q}_m$, we have
\beq
\beta^*\kappa_{v_m^{2r+1}\Sq^1v_m}=t^{2mr}\neq 0\in H^*(BC_2;\F_2)=\F_2[t].
\eeq
It then follows that
\beq
\ka_{v_m^{2r+1}\Sq^1v_m}\neq 0\in H^{2mr}(B\Diff_h(M);\F_2)
\eeq
for every $r\in\msf{Q}_m$. Furthermore for every $j\geq 1$,
\beq
\beta^*\left(\kappa_{v_m^{2r+1}\Sq^1v_m}^j\right)=t^{2mrj}\neq 0.
\eeq
Hence the class $\kappa_{v_m^{2r+1}\Sq^1v_m}$ is not nilpotent. \qed
\subsection{Proof of Theorem \ref{thm:bordism}}
We specialize the preceding construction to $m=2$. Since $M$ is oriented, we have $v_2(M)=w_2(M)$ and $\Sq^1v_2(M)=w_3(M)$. Taking $r=0$ in Lemma \ref{lem:Dold-tautological-class} gives
\beq
\kappa_{w_2w_3}=\kappa_{v_2\Sq^1v_2}=\binom{3}{1}=1\in H^0(B\Diff_h(M);\F_2).
\eeq
Since the degree-zero tautological class is the corresponding Stiefel--Whitney number of $M$, we obtain
\beq
\left\langle w_2(M)w_3(M),[M]\right\rangle=1.
\eeq
In particular, $M$ is not nullbordant and, by Lemma \ref{lem:aspherical-center}, its fundamental group has nontrivial center.

Let now $\mathfrak{N}_*$ denote the unoriented smooth bordism ring. Fix a fake projective plane $F$ as in \cite{mumford}. The manifold $F$ is a quotient of the complex $2$-ball by a torsion-free cocompact discrete subgroup of isometries, and hence is a closed, oriented, smooth, aspherical $4$-manifold. It has the same Betti numbers as $\C P^2$, so $\chi(F)=3$, and hence $\langle w_4(F),[F]\rangle=1$. In particular, $[F]\neq0$ in $\mathfrak{N}_4$. Now for $a\geq1$ and $b\geq0$, set
\beq
M_{a,b}:=\underbrace{M\times\cdots\times M}_{a\ \text{times}}\times\underbrace{F\times\cdots\times F}_{b\ \text{times}}.
\eeq
This is a closed oriented aspherical manifold of dimension $5a+4b$. The elements $t^2$ coming from the $a$ factors of $M^a$ generate a subgroup isomorphic to $\Z^a$ in $\mathrm Z(\pi_1(M_{a,b}))$. The calculation above shows that $[M]\neq0$ in $\mathfrak{N}_*$. Hence $[M_{a,b}]=[M]^a[F]^b\neq0$, because $\mathfrak{N}_*$ has no zero divisors, by Thom's calculations. Thus $M_{a,b}$ is not nullbordant. \qed
\begin{remark}
For $m\equiv 2\pmod4$, Lemma \ref{lem:Dold-tautological-class} with $r=0$ gives
\beq
\left\langle v_m(M)\Sq^1v_m(M),[M]\right\rangle=1.
\eeq
By \cite[Theorem]{lmp}, this characteristic number is equal to the Stiefel--Whitney number
\beq
\left\langle w_2(M)w_{2m-1}(M),[M]\right\rangle.
\eeq
These are more examples of non-nullbordant aspherical manifolds with nontrivial center. 
\end{remark}
\begin{remark}
We briefly explain why the vanishing argument of Hebestreit--Land--L\"uck--Randal-Williams does not extend to $\F_2$-coefficients in our examples. Let $\Gamma=\pi_1(M)$ and let $t\in\Gamma$ be the stable letter. Since the monodromy has order $2$, the element $g=t^2$ is central and has infinite order. The central part of Burghelea's conjecture used in \cite{HLLRW} asserts, with rational coefficients, that $\Gamma/\langle g\rangle$ has finite cohomological dimension with trivial coefficients. This condition is used in \cite[Proposition 6.1.13]{HLLRW} to show that the Euler class of the central extension
\beq
1\to\langle g\rangle\to\Gamma\to\Gamma/\langle g\rangle\to1
\eeq
is nilpotent.

The corresponding statement is false with $\F_2$-coefficients. Indeed, $\Gamma/\langle t^2\rangle\cong\pi_1(W)\rtimes C_2$ contains the subgroup generated by the image of $t$. Restricting the preceding central extension to this subgroup gives
\beq
1\to\langle t^2\rangle\to\langle t\rangle\to C_2\to 1.
\eeq
If $e\in H^2(\Gamma/\langle t^2\rangle;\F_2)$ denotes its mod $2$ Euler class and $x\in H^1(BC_2;\F_2)$ is a generator, then the restriction of $e$ to $BC_2$ is $x^2$. Since $H^*(BC_2;\F_2)=\F_2[x]$, the class $x^2$ is not nilpotent, and hence neither is $e$. In particular, $H^*(\Gamma/\langle t^2\rangle;\F_2)$ is nonzero in arbitrarily high degrees, so the central part of Burghelea's conjecture does not hold for $\Gamma$ with $\F_2$-coefficients. This is one step where the rational vanishing argument of \cite{HLLRW} cannot be applied.
\end{remark}
    \bibliographystyle{amsalpha}
\bibliography{refs}

@incollection {gromov,
    AUTHOR = {Gromov, M.},
     TITLE = {Hyperbolic groups},
 BOOKTITLE = {Essays in group theory},
    SERIES = {Math. Sci. Res. Inst. Publ.},
    VOLUME = {8},
     PAGES = {75--263},
 PUBLISHER = {Springer, New York},
      YEAR = {1987},
      ISBN = {0-387-96618-8},
   MRCLASS = {20F32 (20F06 20F10 22E40 53C20 57R75 58F17)},
  MRNUMBER = {919829},
MRREVIEWER = {Christopher\ W.\ Stark},
       DOI = {10.1007/978-1-4613-9586-7\_3},
       URL = {https://doi-org.pucdechile.idm.oclc.org/10.1007/978-1-4613-9586-7_3},
}

@book {allday-puppe,
    AUTHOR = {Allday, C. and Puppe, V.},
     TITLE = {Cohomological methods in transformation groups},
    SERIES = {Cambridge Studies in Advanced Mathematics},
    VOLUME = {32},
 PUBLISHER = {Cambridge University Press, Cambridge},
      YEAR = {1993},
     PAGES = {xii+470},
      ISBN = {0-521-35022-0},
   MRCLASS = {55N91 (55-02 57-02 57S10)},
  MRNUMBER = {1236839},
MRREVIEWER = {Allan\ Edmonds},
       DOI = {10.1017/CBO9780511526275},
       URL = {https://doi-org.pucdechile.idm.oclc.org/10.1017/CBO9780511526275},
}

@article {lmp,
    AUTHOR = {Lusztig, G. and Milnor, J. and Peterson, F. P.},
     TITLE = {Semi-characteristics and cobordism},
   JOURNAL = {Topology},
  FJOURNAL = {Topology. An International Journal of Mathematics},
    VOLUME = {8},
      YEAR = {1969},
     PAGES = {357--359},
      ISSN = {0040-9383},
   MRCLASS = {57.10},
  MRNUMBER = {246308},
MRREVIEWER = {H.\ Suzuki},
       DOI = {10.1016/0040-9383(69)90021-4},
       URL = {https://doi.org/10.1016/0040-9383(69)90021-4},
}

@article {cwy,
    AUTHOR = {Cappell, S. and Weinberger, S. and Yan, M.},
     TITLE = {Closed aspherical manifolds with center},
   JOURNAL = {J. Topol.},
  FJOURNAL = {Journal of Topology},
    VOLUME = {6},
      YEAR = {2013},
    NUMBER = {4},
     PAGES = {1009--1018},
      ISSN = {1753-8416,1753-8424},
   MRCLASS = {57R67 (57S25)},
  MRNUMBER = {3145148},
MRREVIEWER = {Laurence\ R.\ Taylor},
       DOI = {10.1112/jtopol/jtt023},
       URL = {https://doi.org/10.1112/jtopol/jtt023},
}

@article {gottlieb,
    AUTHOR = {Gottlieb, D. H.},
     TITLE = {A certain subgroup of the fundamental group},
   JOURNAL = {Amer. J. Math.},
  FJOURNAL = {American Journal of Mathematics},
    VOLUME = {87},
      YEAR = {1965},
     PAGES = {840--856},
      ISSN = {0002-9327,1080-6377},
   MRCLASS = {55.40},
  MRNUMBER = {189027},
MRREVIEWER = {George\ McCarty},
       DOI = {10.2307/2373248},
       URL = {https://doi.org/10.2307/2373248},
}

@article {farrell,
    AUTHOR = {Farrell, F. T.},
     TITLE = {The signature and arithmetic genus of certain aspherical
              manifolds},
   JOURNAL = {Proc. Amer. Math. Soc.},
  FJOURNAL = {Proceedings of the American Mathematical Society},
    VOLUME = {57},
      YEAR = {1976},
    NUMBER = {1},
     PAGES = {165--168},
      ISSN = {0002-9939,1088-6826},
   MRCLASS = {57D20 (58G10)},
  MRNUMBER = {407855},
MRREVIEWER = {K.\ H.\ Mayer},
       DOI = {10.2307/2040888},
       URL = {https://doi.org/10.2307/2040888},
}

@article {byun,
    AUTHOR = {Byun, Y.},
     TITLE = {On vanishing of characteristic numbers in {P}oincar\'e{}
              complexes},
   JOURNAL = {Trans. Amer. Math. Soc.},
  FJOURNAL = {Transactions of the American Mathematical Society},
    VOLUME = {348},
      YEAR = {1996},
    NUMBER = {8},
     PAGES = {3085--3095},
      ISSN = {0002-9947,1088-6850},
   MRCLASS = {57R20 (57P10)},
  MRNUMBER = {1322949},
MRREVIEWER = {John\ Bryant},
       DOI = {10.1090/S0002-9947-96-01495-X},
       URL = {https://doi-org.pucdechile.idm.oclc.org/10.1090/S0002-9947-96-01495-X},
}

@article {dj,
    AUTHOR = {Davis, M. W. and Januszkiewicz, T.},
     TITLE = {Hyperbolization of polyhedra},
   JOURNAL = {J. Differential Geom.},
  FJOURNAL = {Journal of Differential Geometry},
    VOLUME = {34},
      YEAR = {1991},
    NUMBER = {2},
     PAGES = {347--388},
      ISSN = {0022-040X},
       DOI = {10.4310/jdg/1214447212},
       URL = {https://doi.org/10.4310/jdg/1214447212},
}

@article {dold,
    AUTHOR = {Dold, A.},
     TITLE = {Erzeugende der {T}homschen {A}lgebra {$\mathfrak{N}$}},
   JOURNAL = {Math. Z.},
  FJOURNAL = {Mathematische Zeitschrift},
    VOLUME = {65},
      YEAR = {1956},
    NUMBER = {1},
     PAGES = {25--35},
      ISSN = {0025-5874},
       DOI = {10.1007/BF01473868},
       URL = {https://doi.org/10.1007/BF01473868},
}

@article {HLLRW,
    AUTHOR = {Hebestreit, F. and Land, M. and L{\"u}ck, W.
              and Randal-Williams, O.},
     TITLE = {A vanishing theorem for tautological classes of aspherical manifolds},
   JOURNAL = {Geom. Topol.},
  FJOURNAL = {Geometry \& Topology},
    VOLUME = {25},
      YEAR = {2021},
    NUMBER = {1},
     PAGES = {47--110},
      ISSN = {1465-3060},
       DOI = {10.2140/gt.2021.25.47},
       URL = {https://doi.org/10.2140/gt.2021.25.47},
}

@article {illman,
    AUTHOR = {Illman, S.},
     TITLE = {Smooth equivariant triangulations of {$G$}-manifolds for
              {$G$} a finite group},
   JOURNAL = {Math. Ann.},
  FJOURNAL = {Mathematische Annalen},
    VOLUME = {233},
      YEAR = {1978},
    NUMBER = {3},
     PAGES = {199--220},
      ISSN = {0025-5831},
       DOI = {10.1007/BF01405351},
       URL = {https://doi.org/10.1007/BF01405351},
}

@article {mumford,
    AUTHOR = {Mumford, D.},
     TITLE = {An algebraic surface with {$K$}\ ample, {$(K\sp{2})=9$},
              {$p\sb{g}=q=0$}},
   JOURNAL = {Amer. J. Math.},
  FJOURNAL = {American Journal of Mathematics},
    VOLUME = {101},
      YEAR = {1979},
    NUMBER = {1},
     PAGES = {233--244},
      ISSN = {0002-9327,1080-6377},
   MRCLASS = {14J25},
  MRNUMBER = {527834},
MRREVIEWER = {Miles\ Reid},
       DOI = {10.2307/2373947},
       URL = {https://doi-org.pucdechile.idm.oclc.org/10.2307/2373947},
}

Mauricio Bustamante\\
Departamento de Matem\'aticas, Pontificia Universidad Cat\'olica de Chile\\
\texttt{mauricio.bustamante@uc.cl}

\end{document}